\documentclass[reqno,english]{amsart}
\usepackage[a4paper,margin=1in]{geometry}
\usepackage{amsmath,amssymb,amsthm,mathtools,booktabs,microtype}
\usepackage[colorlinks=true]{hyperref}
\hypersetup{urlcolor=blue, citecolor=blue}
\hypersetup{
	colorlinks=true,
	linkcolor=red
}

\hypersetup{pdftitle={Rigidity of critical points on the sphere in every dimension N >= 3},
  pdfsubject={Full rigidity, resolvent comparison, and hand-checkable low-dimensional estimates}}
\numberwithin{equation}{section}
\newtheorem{theorem}{Theorem}[section]
\newtheorem{proposition}[theorem]{Proposition}
\newtheorem{lemma}[theorem]{Lemma}
\newtheorem{corollary}[theorem]{Corollary}
\theoremstyle{remark}
\newtheorem{remark}[theorem]{Remark}

\newcommand{\SN}{\mathbb{S}^N}
\newcommand{\Sph}{\mathbb S}
\newcommand*{\ud}{d}
\newcommand{\dd}{\,d}

\newcommand{\intS}{\int_{\SN}}
\newcommand{\iintS}{\iint_{\SN\times\SN}}
\DeclareMathOperator{\tr}{tr}

\title[Rigidity of a $Q$-curvature-type problem on $\mathbb{S}^N$]{Rigidity of a $Q$-curvature-type problem on $\mathbb{S}^N$ in every dimension $N\geq3$}

\author[C. Gui]{Changfeng Gui}
\address[C. Gui]{Department of Mathematics, University of Macau, Taipa, Macau}
\email{changfenggui@um.edu.mo}

\author[Y. Hu]{Yeyao Hu}
\address[Y. Hu]{School of Mathmatics and Statistics, Central South University, Changsha, Hunan, P. R. China}
\email{huyeyao@gmail.com}

\author[T. Li]{Tuoxin Li}
\address[T. Li]{Department of Mathematics, University of Macau, Taipa, Macau}
\email{tuoxinli@um.edu.mo}

\author[J. Wei]{Juncheng Wei}
\address[J. Wei]{Department of Mathematics \\ The Chinese University of Hong Kong \\ Shatin, N.T. \\ Hong Kong}
\email{wei@math.cuhk.edu.hk}

\author[Z. Xie]{Zhantong Xie}
\address[Z. Xie]{Department of Mathematics, University of Macau, Taipa, Macau}
\email{yc57479@um.edu.mo}

\author[Z. Ye]{Zikai Ye}
\address[Z. Ye]{Department of Mathematics\\ Uppsala University\\ 751 06 Uppsala  \\ Sweden
}
\email{zikai.ye@math.uu.se}

\begin{document}



\begin{abstract}
We prove that for every $N\geq 3$, every solution to the $Q$-curvature-type problem $$ \alpha P_N u + (N-1)!\left(1-\frac{e^{Nu}}{\int_{\mathbb{S}^N} e^{Nu}dw}\right)=0 \ \ \ \ \ \mbox{on} \ \mathbb{S}^N $$ is constant, provided that $ \alpha\ge\frac{1}{2}$ and $\alpha \not =1$.  The proof consists of a spectral gap estimate and pointwise estimates of the normalized density $\left(\int_{\mathbb{S}^N} e^{Nu}dw\right)^{-1}e^{Nu}$.
\end{abstract}
\maketitle

\section{Introduction and main results}

Let $dw$ denote the normalized surface measure on $\mathbb{S}^N$ so that $\int_{\mathbb{S}^N}dw=1$. For $\alpha>0$, we consider the $Q$-curvature-type equation:
\begin{equation}\label{eq:main}
 \alpha P_Nu+(N-1)!\left(1-\frac{e^{Nu}}{\intS e^{Nu}\,dw}\right)=0
 \qquad\hbox{on }\SN,
\end{equation}
where
\begin{equation*}
\begin{aligned}
    P_N:=
    \begin{cases}
        \prod_{k=0}^{\frac{N-2}{2}}(-\Delta+k(N-k-1)),&\text{ for }N\text{ even}, \\
        \left(-\Delta+(\frac{N-1}{2})^2\right)^{\frac{1}{2}}\prod_{k=0}^{\frac{N-3}{2}}(-\Delta+k(N-k-1)),&\text{ for }N\text{ odd}.
    \end{cases}
\end{aligned}
\end{equation*}
 Its action on spherical harmonics is
\begin{equation}\label{eq:spectrum}
 P_N1=0,\qquad P_NY_k=\frac{\Gamma(k+N)}{\Gamma(k)}Y_k,
 \qquad Y_k\in\mathcal H_k,\quad k\geq1,
\end{equation}
where $\mathcal H_k$ is the space of spherical harmonics of degree $k$.  In particular, its first two positive eigenvalues are $N!$ and $(N+1)!$. 

When $N=4$, $P_4$ represents the Paneitz operator on $\mathbb{S}^4$ introduced by Paneitz in \cite{Paneitz2008}. When $N$ is even, the operator $P_N$ belongs to the family of conformally invariant operators $P_{N,g,k}$ with $k=\frac{N}{2}$, constructed by Graham, Jenne, Mason, and Sparling on a Riemannian manifold $(M^N, g)$ in \cite{GJMS1992}, called the GJMS operator. When $N$ is odd, $P_N$ is a nonlocal elliptic pseudo-differential operator of order $N$ and the square root is understood in the spectral sense.
We refer to \cite{CLY2019, CM2023, ChangYang1995, ChangYang1997, DHL2000,DM2008, FG2013, GurMal2015, LiXiong2019, Mal2006} and the references therein for results and background on $Q$-curvature problems and GJMS operators.  

Equation \eqref{eq:main} has a variational structure. In fact, it is the Euler--Lagrange equation of the functional
\begin{equation}\label{eq:Beckner-functional}
J_{\alpha,N}(u):=\frac{\alpha}{2}\int_{\Sph^N}uP_Nu\,\ud w+(N-1)!\int_{\Sph^N}u\,\ud w-\frac{(N-1)!}{N}\log\!\left(\int_{\Sph^N}e^{Nu}\,\ud w\right).
\end{equation}
The well-known Beckner's inequality \cite{Beckner1993} states that $\inf J_{1,N}(u)=0$.
In addition, A. Chang and P. Yang \cite{ChangYang1995} proved that, if $u$ belongs to the zero center-of-mass set
\begin{equation*}
\mathcal{L}_{N}=\left\{u\in H^{\frac{N}{2}}(\mathbb{S}^N)\ :\ \int_{\mathbb{S}^N}e^{Nu} x_j dw=0,\ j=1,\dots, N+1 \right \},
\end{equation*}
then for any $\alpha\geq \frac{1}{2}$, there exists a constant $C(\alpha,N)\geq0$ such that 
\begin{equation*}
J_{\alpha,N}(u)\geq -C(\alpha,N)    
\end{equation*}
for any $u\in \mathcal{L}_{N}$. This motivates the study of the constrained Euler--Lagrange equation of $\mathcal{J}_{\alpha,N}$ in the constraint set $\mathcal{L}_N$.

A critical point of $\mathcal{J}_{\alpha,N}$ under the zero center-of-mass constraint satisfies the following $Q$-curvature-type equation on $\mathbb{S}^N$:
\begin{equation*}
\alpha P_N u+(N-1)!(1-\frac{e^{Nu}}{\int_{\mathbb{S}^N}e^{Nu}dw})=\sum_{i=1}^{N+1}a_i x_i e^{Nu} \ \mbox{on} \ \mathbb{S}^N
\end{equation*}
for some constants $a_i$, $i=1,\dots, N+1$. When $\alpha\neq1$, by applying the Kazdan–Warner identity, one can show that all Lagrange multipliers $a_i$ vanish. 
The converse is also true. In fact, by \cite[Proposition 1.5]{GHW2022}, we have 
\begin{equation*}\label{u centre}
\int_{\mathbb{S}^N} u x_jdw=\int_{\mathbb{S}^N}  e^{Nu}x_jdw=0, \quad j=1,\dots, N+1.    
\end{equation*}
We refer to Chang-Yang \cite{ChangYang1995} and Wei-Xu \cite{WX1998} for more details. 

\medskip

On $\mathbb{S}^2$, \eqref{eq:main} is known as the Nirenberg problem:
\begin{equation*}
-\alpha \Delta u + 1-  \frac{e^{2u}}{\int_{\mathbb{S}^2} e^{2u}}=0 \ \ \mbox{on} \ \mathbb{S}^2.
\end{equation*}
This problem has been extensively studied over the past four decades. See \cite{ChangYang1987, ChangYang1988, JinLiXiong2017} and the references therein. For $\alpha<1$, Feldman, Froese, Ghoussoub, and Gui \cite{FFGG1998} first established the rigidity for axially symmetric functions when $\alpha >\frac{16}{25}-\epsilon$. Gui and Wei \cite{GW2000} subsequently proved the optimal constant $\alpha=\frac{1}{2}$ in the axially symmetric case.
Finally, Gui and Moradifam \cite{GM2018} established the sphere covering inequality to prove that every solution is  axially symmetric if $\alpha\ge \frac{1}{2}$, thereby completing the classification on $\mathbb{S}^2$. For more related results, we refer readers to \cite{ChangGui202, ChangHang2022, GHM2020, SSTW2019}.

For the higher-dimensional equation \eqref{eq:main} on $\mathbb{S}^N$, Wei and Xu \cite{WX1998} obtained rigidity for $\alpha<1$ sufficiently close to $1$, while Zhang \cite{Zhang2025} treated $\alpha>1$. Since direct treatment of \eqref{eq:main} seems difficult, in view of the work of Ghoussoub-Lin \cite{GL2010} and Gui-Moradifam \cite{GM2018}, it is natural to consider axially symmetric functions as a first step in higher dimensions. Sharp results for axially symmetric solutions in dimensions $4$, $6$, and $8$ were established in \cite{GHX2021,GHW2022,GLWY2025,LWY2022}. For odd $N\geq 3$, few results are known about the rigidity. One of the difficulties is the non-locality of the operator $P_N$ when $N$ is odd. 


 The recent work \cite{GLWY} proves the sharp centered Beckner inequality and establishes rigidity for \emph{stable} constrained critical points. In this paper, we prove the rigidity for \emph{all} solutions of \eqref{eq:main} in every dimension $N\geq 3$.
\begin{theorem}\label{thm:main}
Let $N\geq3$ be an integer and let $\alpha\ge\frac12$, $\alpha\neq 1$.  Then every weak solution $u\in H^{\frac{N}{2}}(\mathbb{S}^N)$ of \eqref{eq:main} is constant. 
\end{theorem}
\begin{remark}
The proof of regularity in \cite[Proposition 2.2]{GLWY} also works if we only assume $u\in H^{\frac{N}{2}}(\mathbb{S}^N)$ is a weak solution of \eqref{eq:main}. Hence, we can always assume $u\in C^{\infty}(\mathbb{S}^N)$ in the rest of this paper.
\end{remark}
\begin{remark}
   When $\alpha=1$, there is a family of solutions $u(x)=-\log(1-a \cdot x)$ to \eqref{eq:main} for any $|a|<1$. This is why we have to exclude $\alpha=1$ in Theorem \ref{thm:main}. In addition, this family of solutions are also optimizers of Beckner's inequality.
\end{remark}

Previous papers \cite{GHX2021, GHW2022, GLWY2025, GW2000, LWY2022} concerning axially symmetric functions rely heavily on a priori numerical observations. Hence, this dimension-dependent strategy does not extend to general dimensions $N$. In \cite{GLWY}, the authors established identities for general functions and chose an appropriate test function in the second variation of the stable critical point. 

In this paper, we will use some preliminary integral identities from  \cite{GLWY}. The proof studies the normalized density $g=\frac{e^{Nu}}{\intS e^{Nu} dw}$. Two pointwise estimates for the maximum $D=\max_{\SN}g$ lead to a contradiction when $g$ is non-constant. First, vanishing of the degree-one component and the spectral gap at degree two give a lower bound for $D$. Second, the logarithmic Green representation of $u$ yields an upper bound.

This paper is organized as follows. In Section \ref{sec prelim}, we recall some integral identities and estimate the eigenvalues of the covariance matrix.  Section~\ref{sec est D} gives pointwise estimates on the maximum of the density function $g$. In Section \ref{sec pf main}, we combine these estimates to conclude the proof of Theorem \ref{thm:main}.

\section{Integral identities and the eigenvalues of \texorpdfstring{$M$}{M}}\label{sec prelim}

Throughout the proof, we use the notations
\begin{equation}\label{def:notation}
 \gamma=\intS e^{Nu}\,dw,\qquad
 g=\frac{e^{Nu}}\gamma,\qquad d\mu=g\,dw,\qquad
 M_{ij}=\intS\xi_i\xi_j\,d\mu(\xi).
\end{equation}
Thus $\intS g\,dw=1$, $M$ is positive definite, and $\tr M=1$.

\begin{lemma}[Proposition 1.5 in \cite{GHW2022}]\label{lem:center}
Let $u$ be a solution of \eqref{eq:main} with $\alpha\neq 1$. Then
	\begin{equation}\label{eq:center}
 \intS\xi g\,dw(\xi)=0,\qquad \intS\xi u\,dw(\xi)=0.
\end{equation}
\end{lemma}

\begin{lemma}\label{lem:green}
Let $u$ be a solution of \eqref{eq:main} and set $\bar u=\intS u\,dw$. Then
\begin{equation}\label{eq:green}
 u(\xi)-\bar u=-\frac1\alpha\intS
 \log(1-\xi\cdot\eta)(g(\eta)-1)\,dw(\eta),
\end{equation}
and
\begin{equation}\label{eq:gradient}
 \nabla u(\xi)=\frac1\alpha\intS
 \frac{\eta-(\xi\cdot\eta)\xi}{1-\xi\cdot\eta}\,d\mu(\eta).
\end{equation}
\end{lemma}
\begin{proof}
\eqref{eq:green} follows directly from Lemma 3.2 and equation (3.7) of \cite{GLWY}. 

Differentiating \eqref{eq:green} and using $\nabla_\xi(\xi\cdot\eta)=\eta-(\xi\cdot\eta)\xi$, we obtain
\begin{equation*}
    \nabla u(\xi)=\frac1\alpha\intS \frac{\eta-(\xi\cdot\eta)\xi}{1-\xi\cdot\eta}(g(\eta)-1)\dd w(\eta).
\end{equation*}
To justify the differentiation, replace $1-\xi\cdot\eta$ by $1-\xi\cdot\eta+\varepsilon$ and then let $\varepsilon$ tend to zero. 
Moreover,
\begin{equation*}
    \intS\frac{\eta-(\xi\cdot\eta)\xi}{1-\xi\cdot\eta}\dd w(\eta)=0
\end{equation*}
by rotational invariance. Thus we obtain \eqref{eq:gradient}.
\end{proof}

Fix a unit vector $v\in\mathbb R^{N+1}$ and, for $\xi,\eta\in\SN$, write
\begin{equation*}
 x=v\cdot\xi,\qquad y=v\cdot\eta,\qquad s=\xi\cdot\eta.
\end{equation*}

\begin{lemma}\label{lem:chord}
For every $\varphi\in C^1([-1,1])$,
\begin{align}\label{eq:chord}
 &\iintS\frac{(x-y)(\varphi(x)-\varphi(y))}{1-s}\,d\mu(\xi)d\mu(\eta)\notag\\
 &\hspace{1cm}=2(1-\alpha)\intS x\varphi(x)\,d\mu
       +\frac{2\alpha}{N}\intS(1-x^2)\varphi'(x)\,d\mu.
\end{align}
\end{lemma}
\begin{proof}
	Since $\nabla g=Ng\nabla u$, integration by parts gives
		\begin{equation}\label{int vphi 1}
			\begin{aligned}
					N\intS\varphi(x)\langle\nabla u,\nabla x\rangle\,d\mu
				=&\intS\varphi(x)\langle\nabla g,\nabla x\rangle\,dw(\xi)\\
				=&-\intS\text{div}(\varphi(x)\nabla x)\,d\mu\\
				=&\intS\bigl(Nx\varphi(x)-(1-x^2)\varphi'(x)\bigr)\,d\mu.
			\end{aligned}
		\end{equation}
		By \eqref{eq:gradient},
		\begin{equation*}
			\alpha\langle\nabla u(\xi),\nabla x\rangle
			=\intS\frac{y-xs}{1-s}\,d\mu(\eta)
			=x+\intS\frac{y-x}{1-s}\,d\mu(\eta).
		\end{equation*}
		Multiply the above formula by $\varphi(x)$ and integrate. Using \eqref{int vphi 1}, we obtain
		\begin{equation*}
	\frac{\alpha}{N}	\intS\bigl(Nx\varphi(x)-(1-x^2)\varphi'(x)\bigr)\,d\mu(\xi)	=\intS x\varphi(x)d \mu(\xi)+\iintS\frac{y-x}{1-s}\varphi(x)\,d\mu(\xi)\,d\mu(\eta).
		\end{equation*}
		 Interchanging $\xi$ and $\eta$ in the double integral and combining the two expressions prove \eqref{eq:chord}.
\end{proof}

Let $v$ be a unit eigenvector of $M$ with eigenvalue $m$. For $\xi,\eta\in \mathbb{S}^N$, write $x=v\cdot \xi$, $y=v\cdot \eta$ and $s=\xi\cdot \eta$. Define the kernel
\begin{equation}
D_v(\xi,\eta):=
\begin{cases}
\frac{(x-y)^2}{1-s},
  &\qquad \xi\neq \eta,\\
0,&\qquad \xi=\eta.
\end{cases}
\end{equation}
Since
\begin{equation*}
(x-y)^2\leq |\xi-\eta|^2=2(1-s),
\end{equation*}
we have $0\leq D_v\leq 2$. 

With the help of $D_v$, in the following, we establish an estimate of the largest eigenvalue of $M$, which refines the estimate in \cite[Theorem 1.4]{GLWY}.

\begin{proposition}\label{prop:M}
Let $0\leq \delta<1$ with $\delta\leq g_*:=\min_{\mathbb{S}^N} g$. For $m=\lambda_{max}(M)$, we have $m\geq \frac{1}{N+1}$ and
\begin{equation}\label{eq:Mbound}
F_N(m,\delta):=N\left(m-\frac{\delta}{N+1}\right)^2(1-\delta)-\left(\frac{1}{2}-\frac{\delta}{N+1}\right)((N+1)m-1)\left(1+m-\frac{N+2}{N+1}\delta \right)\geq 0.
\end{equation}
The inequality is strict when $\delta=0$. In particular, we have $m<\frac{1}{N-1}$.
\end{proposition}
\begin{proof}
The lower bound of $m$ follows from $m=\lambda_{max}(M)$ and $\text{tr} M=1$.

In the following, we consider 
\begin{equation*}
g_{\delta}:=\frac{g-\delta}{1-\delta}\geq 0,
\end{equation*}
the measure $d\mu_{\delta}:=g_{\delta}dw$, and the matrix $M_{\delta}$ defined by
\begin{equation*} M_{\delta,ij}:=\intS\xi_i\xi_j\,d\mu_{\delta}(\xi)=\frac{1}{1-\delta}\left(M_{ij}-\frac{\delta_{ij}}{N+1}\right).
\end{equation*}

Then we have
\begin{equation*}
\int_{\mathbb{S}^N}d\mu_{\delta}=1,\qquad \int_{\mathbb{S}^N}\xi d\mu_{\delta}=0.
\end{equation*}

Recall from \cite[Section 3]{GLWY} that
\begin{equation*}
\iintS D_v d\mu d\mu=2(m-\alpha t)
\end{equation*}
with $t=\frac{(N+1)m-1}{N}$,
\begin{equation*}
\intS D_v(\xi,\eta) dw(\eta) =\frac{1}{N}+\frac{N-1}{N}x^2
\end{equation*}
and
\begin{equation*}
\iintS D_v dw dw=\frac{2}{N+1}.
\end{equation*}

It follows that
\begin{equation}\label{eq:Dvmudelta}
\iintS D_v d\mu_{\delta} d\mu_{\delta}=2(m_{\delta}-\alpha_{\delta} t_{\delta})
\end{equation}
with
\begin{equation*}
m_{\delta}:=\frac{1}{1-\delta}\left(m-\frac{\delta}{N+1}\right),\qquad \alpha_{\delta}:=\frac{1}{1-\delta}\left(\alpha-\frac{\delta}{N+1}\right),\qquad t_{\delta}:=\frac{(N+1)m_{\delta}-1}{N}.
\end{equation*}

 Applying the argument of \cite[Lemma 3.7]{GLWY} to $d\mu_{\delta}$ gives
\begin{equation*}
\iintS D_v d\mu_{\delta} d\mu_{\delta}\geq \frac{2m_{\delta}}{1+m_{\delta}}.
\end{equation*}
Combining this bound with \eqref{eq:Dvmudelta} yields 
\begin{equation*}
Nm_{\delta}^2\geq \alpha_{\delta}((N+1)m_{\delta}-1)(1+m_{\delta}),
\end{equation*}
Since $\alpha\geq \frac{1}{2}$, we obtain \eqref{eq:Mbound}. 

The inequality is strict when $\delta=0$ since $g>0$. Consequently, we have
\begin{equation*}
0<F_N(m,0)=\frac{1}{2}(m-1)((N-1)m-1).
\end{equation*}
Since $tr (M)=1$ and $M$ is positive definite, this yields $m<\frac{1}{N-1}$.
\end{proof}

With the help of this proposition, we can obtain the following upper bound for $g_*$.
\begin{lemma}\label{lem:g_*upper}
If $m\geq \frac{1}{N}$, then 
    \begin{equation*}
    g_*<\frac{N}{3}(1-(N-1)m).
    \end{equation*}
 
\begin{proof}
Set $\delta=\frac{N}{3}(1-(N-1)m)$. Suppose on the contrary that $g_*\geq \delta$. By Proposition \ref{prop:M}, we have $0<\delta\leq \frac{1}{3}$ and $\delta \leq g_*$. Then a direct computation gives
\begin{equation*}
F_N(\frac{N-3\delta}{N(N-1)},\delta)=-\frac{\delta}{2N^2(N-1)(N+1)^2}R_N(\delta),
\end{equation*}
where $F_N$ is defined in Proposition \ref{prop:M} and
\begin{equation*}
R_N(\delta)=(2N^4+4N^3+12N^2+24N+18)\delta^2-(3N^4+4N^3+16N^2+24N+9)\delta+(N^4+5N^2+6N).
\end{equation*}

One can check that $R_N$ is decreasing on $0<\delta \leq \frac{1}{3}$ and 
\begin{equation*}
R_N(\frac{1}{3})=\frac{2N^2(N-2)^2+N^2+6N-9}{9}>0.
\end{equation*}
It follows that $F_N(\frac{N-3\delta}{N(N-1)},\delta)<0$, which is a contradiction to Proposition \ref{prop:M}.

\end{proof}
\end{lemma}

\section{Estimates at a maximum point}\label{sec est D}

To prove Theorem \ref{thm:main}, from now on, we assume for contradiction that $u$ is non-constant. Then
\begin{equation}\label{eq:D}
 D:=\max_{\SN}g>1.
\end{equation}

\subsection{A lower bound for $D$}
The spectral gap yields the following lower bound for $D$.
\begin{lemma}\label{lem:amplitude}
If $u$ is a non-constant solution of \eqref{eq:main}, then
\begin{equation}\label{eq:amplitude}
 \frac{D-1}{\log D}>\alpha(N+1)\geq\frac{N+1}{2}.
\end{equation}
\end{lemma}
\begin{proof}
We write \eqref{eq:main} equivalently as
$$\alpha P_N V=N!(g-1),$$
where 
\begin{equation*}
V=\log g-\intS \log g dw.
\end{equation*}

By \eqref{eq:center}, both $g-1$ and $V$ contain no spherical harmonics of degree one. Then we can expand
\begin{equation*}
V=\log g-\intS \log g dw=:\sum_{k=2}^{\infty}V_k,
\end{equation*}
where $V_k\in \mathcal{H}_k$ is spherical harmonic of degree $k$.

Consequently, by \eqref{eq:spectrum}, for $p_k=\frac{(k+N-1)!}{(k-1)!}$, we have
\begin{equation}\label{eq:spectralgap}
	\begin{aligned}
			\intS(g-1)^2\,dw = &\sum_{k\ge 2}(\frac{\alpha}{N!})^2p_k^2 \|V_k\|_2^2\ge \alpha(N+1)\sum_{k\ge 2}\frac{\alpha}{N!}p_k \|V_k\|_2^2\\
			=&\alpha(N+1)\intS(g-1)\log g\,dw.
	\end{aligned}
\end{equation}

The function $\frac{t-1}{\log t}$ is strictly increasing for $t>0$. Therefore
\begin{equation*}
(g-1)^2\leq \frac{D-1}{\log D}(g-1)\log g.
\end{equation*}
Here the inequality is strict in $\{0<g<1\}$, which is empty since $u$ is non-constant and $\intS gdw=1$.
    
Integrating this inequality and comparing with \eqref{eq:spectralgap} yields \eqref{eq:amplitude}.
\end{proof}

\medskip

\subsection{An upper bound for $D$}
In the following part of this section, we derive an upper bound for $D$ from the logarithmic Green potential.

For $s>0$, define the probability measure
\begin{equation}\label{eq:nu}
 d\nu_s(t)=\frac{(1-t^2)^{s/2-1}\,dt}{\int_{-1}^{1}(1-t^2)^{s/2-1}\,dt}.
\end{equation}
Then we have
\begin{equation}\label{eq:latitude-moments}
 \int_{-1}^1\,d\nu_s=1,\qquad
 \int_{-1}^1t\,d\nu_s(t)=0.
\end{equation}

We also define
\begin{equation}\label{eq:Bdef}
\mathcal{B}(s):=\int_{-1}^{1}-\log (1-t)d\nu_s(t)=\int_{0}^{1}\frac{t^{s-1}}{1+t}dt.
\end{equation}

For a unit vector $v\in \mathbb{S}^N$, we define the logarithmic potential
\begin{equation*}
		T_N(v)=-\intS\log(1-v\cdot\xi)\,d\mu(\xi).
\end{equation*}

Then \eqref{eq:green} gives
\begin{equation}\label{eq:logD}
\log g(v)=N \bar{u}-\log \gamma+\frac{N}{\alpha}(T_N(v)-\mathcal{B}(N)).
\end{equation}

By Jensen's inequality, we have
\begin{equation}\label{eq:Nuloggamma}
N \bar{u}-\log \gamma=\intS \log g dw\leq \log \intS gdw=0.
\end{equation}

Meanwhile, the beta integral gives
\begin{equation*}
	\mathcal{B}(N)=\psi(N)-\psi(\frac{N}{2})-\log 2,
\end{equation*}
where $\psi(x):=\frac{\Gamma'(x)}{\Gamma(x)}$ is the digamma function.

It suffices to estimate $T_N$. The following proposition gives an upper bound that is uniform in $v$. 
\begin{proposition}\label{prop:TNbound}
Let $$c=N\frac{1-m}{1+m}, \qquad\theta=1-\frac{1-(c+1)m}{c+2}=1-\frac{(1-m)(1-(N-1)m)}{(c+2)(1+m)}\in (0,1).$$
Let
\begin{equation*}
    H_N(m)=\theta \mathcal{B}(c)-(1-\theta)(\log 2-\frac{1}{2}).
\end{equation*}
Then for any $v\in \mathbb{S}^N$, we have
\begin{equation}\label{eq:TNbound}
	0\leq T_N(v)<H_N(m).
\end{equation} 
\end{proposition}
\begin{proof}
By Jensen's inequality and \eqref{eq:center}, we have
\begin{equation*}
T_N(v)\geq -\log \intS (1-v\cdot \xi)d\mu (\xi)=0.
\end{equation*}

For any $j\geq 1$, define
\begin{equation*}
\mathcal{M}_j(v)=\intS (1+v\cdot \xi)^j d\mu (\xi)
\end{equation*}
and $\mathcal{M}_j^*=\max_{v\in \mathbb{S}^N}\mathcal{M}_j(v)$.

By \eqref{eq:center}, we have $\mathcal{M}_1^*=1$. For $j\geq 2$, we choose $v_j\in \mathbb{S}^N$ such that $\mathcal{M}_j(v_j)=\mathcal{M}_j^*$. For simplicity, in the following, we write  $v=v_j$. Then for $x=v\cdot \xi$ and any $e\in \mathbb{S}^N$ with $e\cdot v=0$, we have
\begin{equation*}
0=\partial_e \mathcal{M}_j(v)=j\intS (e\cdot \xi)(1+x)^{j-1}d\mu(\xi)=j e\cdot Z,
\end{equation*}
where
\begin{equation}\label{eq:ZAv}
Z:=\intS \xi (1+x)^{j-1 }d\mu (\xi).
\end{equation}
It follows that $Z=Av$ with
\begin{equation*}
A:=\intS x(1+x)^{j-1}d\mu(\xi).
\end{equation*}

Write
\begin{equation*}
W(x,y):=(x-y)((1+x)^{j-1}-(1+y)^{j-1}).
\end{equation*}
Then $W\geq 0$ and the equality holds if and only if $x=y$. Let $m_v=v^T Mv$. By \eqref{eq:center}, we have
\begin{equation}\label{eq:W}
\iintS W d\mu d\mu=2A
\end{equation}
and
\begin{equation}\label{eq:W(1-s)}
\iintS W(1-s) d\mu d\mu=2A+2(Mv)\cdot Z=2(1+m_v)A.
\end{equation}
In particular, we have $A>0$.

Set
\begin{equation*}
B=(j-1)\intS (1-x^2)(1+x)^{j-2}d\mu.
\end{equation*}

Applying Lemma \ref{lem:chord} with $\varphi(x)=(1+x)^{j-1}$, we obtain that
\begin{equation}\label{eq:W/(1-s)}
\iintS \frac{W}{1-s} d\mu d\mu=2(1-\alpha)A+\frac{2\alpha}{N}B.
\end{equation}

By Cauchy--Schwarz inequality, we have
\begin{equation*}
\left(\iintS W\,d\mu(\xi)\,d\mu(\eta)\right)^2<\left(\iintS \frac{W}{1-s}d\mu(\xi) d\mu(\eta)\right)\left(\iintS (1-s)W\,d\mu(\xi)\,d\mu(\eta)\right).
\end{equation*}
Here the equality cannot occur since $g>0$. Substituting \eqref{eq:W}, \eqref{eq:W(1-s)} and \eqref{eq:W/(1-s)} into this inequality gives
\begin{equation*}
1<(1+m_v)(1-\alpha+\frac{\alpha}{N}\frac{B}{A}).
\end{equation*}
Since $m_v\leq m$ and $\alpha\geq \frac{1}{2}$, we have
\begin{equation}\label{eq:B/A}
\frac{B}{A}>N\left(1-\frac{m_v}{\alpha (1+m_v)}\right)\geq \frac{1-m}{1+m}N=c.
\end{equation}

On the other hand, by the definition of $A$ and $B$, we have $A=\mathcal{M}_j(v)-\mathcal{M}_{j-1}(v)$ and $B=(j-1)(2\mathcal{M}_{j-1}(v)-\mathcal{M}_{j}(v))$.
It follows from \eqref{eq:B/A} that for any $j\geq 2$,
\begin{equation}\label{Mj recur}
\mathcal{M}_{j}^*<\frac{2(j-1)+c}{j-1+c}\mathcal{M}_{j-1}^*.
\end{equation}

Set
\begin{equation*}
Q_j=\int_{-1}^{1}(1+t)^jd\nu_c(t)=2^j \frac{(\frac{c}{2})_j}{(c)_j},
\end{equation*}
where $(c)_j$ is the Pochhammer symbol. Then $Q_j$ satisfies the equality in \eqref{Mj recur}. Moreover, we have  $\mathcal{M}_1^*=Q_1=1$, and
\begin{equation*}
    \mathcal{M}_2^*=1+m,\qquad Q_2=\frac{c+2}{c+1}.
\end{equation*}
By definition, $\theta=\frac{\mathcal{M}_2*}{Q_2}$. Then we can prove inductively that
\begin{equation*}
\mathcal{M}_j^*\le\theta Q_j, \quad j\ge 2,
\end{equation*}
with strict inequality for $j\geq 3$.

Since for any $-1\leq t< 1$, 
\begin{equation*}
-\log (1-t)=-\log 2+\sum_{j=1}^{\infty}\frac{(1+t)^j}{j2^j},
\end{equation*}
we obtain that
\begin{equation*}
T_N(v)=-\log 2+\intS\sum_{j=1}^{\infty}\frac{(1+v\cdot \xi)^j}{j2^j} d\mu(\xi)<-\log 2+\frac{1}{2}+\theta\sum_{j=2}^\infty\frac{Q_j}{j2^j}=\theta\mathcal{B}(c)-(1-\theta) \left(\log 2-\frac{1}{2}\right),
\end{equation*}
which is exactly \eqref{eq:TNbound}.

\end{proof}

The minimum of $g$ gives the following  lower bound for $T_N(v)$.
\begin{lemma}\label{lem TN low}
    Recall that $g_*=\min_{\mathbb{S}^N} g$. Then for every $v\in\mathbb{S}^N$,
    \begin{equation}\label{T low}
        T_N(v)\ge g_*\mathcal{B}(N).
    \end{equation}
\end{lemma}
\begin{proof}
    If $g\equiv 1$, then $g_*=1$, and equality in \eqref{T low} holds. Otherwise, we write
    \begin{equation*}
        d\mu=g_*dw+(1-g_*)d\hat\mu,\qquad d\hat{\mu}=\frac{g-g_*}{1-g_*}dw.
    \end{equation*}
    Then $\hat{\mu}$ is a probability measure and its center-of-mass is the origin. Hence, Jensen's inequality gives
    \begin{equation*}
        -\int_{\mathbb{S}^N}\log (1-v\cdot\xi)d\hat{\mu}\ge 0.
    \end{equation*}
    Therefore,
   \begin{equation*}
		T_N(v)=-\intS\log(1-v\cdot\xi)\,d\mu(\xi)\ge -\intS\log(1-v\cdot\xi)\,g_*dw=g_*\mathcal{B}(N).
\end{equation*}
\end{proof}

As a corollary, we have the following estimates on $D$ .
\begin{corollary}\label{cor:logDupper}
Let $N\geq 3$ and $\alpha\geq \frac{1}{2}$. If $u$ is a non-constant solution of \eqref{eq:main}, then 
\begin{equation}\label{est logD 1}
\log D<2N(H_N(m)-\mathcal{B}(N)),
\end{equation}
and
\begin{equation}\label{est logD 2}
    \log D<\log g_*+2N(H_N(m)-g_*\mathcal{B}(N)).
\end{equation}
\end{corollary}
\begin{proof}
We choose $v_{\pm}\in \mathbb{S}^N$ such that $g(v_+)=D$ and $g(v_-)=g_*$. By \eqref{eq:logD}, \eqref{eq:Nuloggamma}, Proposition \ref{prop:TNbound} and $\alpha \geq \frac{1}{2}$, we have 
\begin{equation*}
0<\log D=\log g(v_+)\leq 2N(T_N(v_+)-\mathcal{B}(N))<2N(H_N(m)-\mathcal{B}(N)).
\end{equation*}   

Subtracting \eqref{eq:logD} at $v_-$ from the same identity at $v_+$ gives
\begin{equation*}
0<\log \frac{D}{g_*}=\frac{N}{\alpha}(T_N(v_+)-T_N(v_-)).
\end{equation*}
Then Proposition \ref{prop:TNbound}, Lemma \ref{lem TN low} and $\alpha\ge\frac{1}{2}$ yield \eqref{est logD 2}.
\end{proof}

\section{Proof of Theorem \ref{thm:main}}\label{sec pf main}
Finally, in this section, we give the proof of Theorem \ref{thm:main}.
\begin{proof}[Proof of Theorem \ref{thm:main}]
Suppose that $g$ is nonconstant. We first consider the case $N\geq 4$. Since $m<\frac{1}{N-1}$, $c=N\frac{1-m}{1+m}>N-2$. Using $H_N(m)<\mathcal{B}(c)$ and \eqref{eq:Bdef}, \eqref{est logD 1} gives
\begin{equation*}
    \begin{aligned}
        \log D<&2N(\mathcal{B}(c)-\mathcal{B}(N))<2N(\mathcal{B}(N-2)-\mathcal{B}(N))\\
        =&2N\int_0^1\frac{t^{N-3}-t^{N-1}}{1+t}dt=\frac{2N}{(N-2)(N-1)}\\
        \le&\frac{4}{3}.
    \end{aligned}
\end{equation*}
Since $\frac{t-1}{\log t}$ is strictly increasing for $t>0$, Lemma \ref{lem:amplitude} yields
\begin{equation*}
    \frac{5}{2}\le \alpha(N+1)<\frac{D-1}{\log D}<\frac{3(e^{\frac{4}{3}}-1)}{4}<\frac{9}{4},
\end{equation*}
a contradiction. Thus $g\equiv1$.

\medskip

It remains to consider the case $N=3$. According to the notations in Proposition \ref{prop:TNbound}, we define
\begin{equation*}
    b=\mathcal{B}(3)=\log 2-\frac{1}{2},\quad c(m)=3\frac{1-m}{1+m}.
\end{equation*}
\begin{equation*}
\theta(m)=1-\frac{(1-m)(1-2m)}{5-m},\quad
      H(m)=\theta(m) \mathcal{B}(c)-(1-\theta(m))b.
\end{equation*}

By Lemma \ref{lem:amplitude}, 
\begin{equation*}
    \frac{D-1}{\log D}>4\alpha\ge 2.
\end{equation*}
Since $e^{\frac{5}{4}}<\frac{7}{2}$, we obtain that
\begin{equation*}
    \log D>\frac{5}{4}.
\end{equation*}
We prove the opposite inequality by splitting according to the value of $m$.

If $m\le \frac{1}{3}$, then $c(m)\ge \frac{3}{2}$, and
\begin{equation*}
    \frac{20}{21}-\theta(m)=\frac{2(3m-1)(7m-8)}{21(5-m)}\ge 0.
\end{equation*}
Consequently
\begin{equation*}
    H(m)-b\leq \frac{1}{21}(20\mathcal{B}(\frac{3}{2})-22b).
\end{equation*}

Thus, \eqref{est logD 1} implies that
\begin{equation*}
    \log D\le \frac{2}{7}(20\mathcal{B}(\frac{3}{2})-22b)<\frac{5}{4},
\end{equation*}
a contradiction.

Now we suppose $m>\frac{1}{3}$. Set $r=1-2m\in (0,\frac{1}{3})$. Then Lemma \ref{lem:g_*upper} implies $g_*<r$. By \eqref{est logD 2}, we have
\begin{equation*}
    \log D<\log g_*+6(H(m)-g_*b).
\end{equation*}
Since $b<\frac{1}{2}$, the function $t\mapsto \log t-6bt$ is strictly increasing on $(0, \frac{1}{3}]$. Therefore
\begin{equation*}
    \log D<\log r+6(H(m)-br)<\frac{5}{4},
\end{equation*}
which is also a contradiction.

In conclusion, in every case, we proved that $g$ is the constant function $1$, and hence $u$ is constant. This proves Theorem \ref{thm:main}.
\end{proof}

\section*{Acknowledgements}
 The research of  C. Gui is supported by NSFC Key Program (Grant No.12531010), University of Macau research grants CPG2024-00016-FST, CPG2025-00032-FST, CPG2026-00027-FST, SRG2023-00011-FST, MYRGGRG2023-00139-FST-UMDF, UMDF Professorial Fellowship of Mathematics, Macao SAR FDCT 0003/2023/RIA1 and Macao SAR FDCT 0024/2023/RIB1.
 The research of Y. Hu was partially supported by the National Key R\&D Program of China (2022YFA1006800), the National Natural Science Foundation of China (Grant No.~12471205), the Science and Technology Innovation Program of Hunan Province (2025RC1005), the Hunan Basic Science Research Center for Mathematical Analysis (2024JC2002), and the Key Project of the Hunan Provincial Education Department (24A0004).
The research of J. Wei is partially supported by General Research Grant of HKSAR GRF 14309824. The authors acknowledge the use of AI tools. All mathematical arguments and proofs in the final manuscript were checked and written by the authors.

\end{document}